\documentclass[12pt,a4paper]{article}
\usepackage{amsfonts,amssymb,amsmath}
\usepackage{theorem}
\input diagrams
{\theorembodyfont{\rm}
  \newtheorem{defi}{Definition}[section]
  
  \newtheorem{exa}[defi]{Example}
  \newtheorem{exas}[defi]{Examples}
}
  \newtheorem{lem}[defi]{Lemma}
  \newtheorem{prop}[defi]{Proposition}
  \newtheorem{thm}[defi]{Theorem}
  \newtheorem{cor}[defi]{Corollary}

\newcommand{\para}{\par\vspace{1ex plus0.2ex minus0.2ex}
\refstepcounter{defi}\textbf{\arabic{section}.\arabic{defi}}\, }

\newcommand{\PP}{{\mathbb P}}

\newcommand{\ZZ}{{\mathbb Z}}

\newcommand{\cD}{{\mathcal D}}

\newcommand{\cS}{{\mathcal S}}

\newcommand{\T}{{\mathrm{T}}}

\newcommand{\id}{{\mathrm{id}}}
\newcommand{\dis}%{{\mathrel{\vartriangle}}}
                 {{\mathrel{\scriptstyle{\triangle}}}}

\newcommand{\eps}{{\varepsilon}}

\newcommand{\GL}{{\mathrm{GL}}}

\newcommand{\E}{{\mathrm{E}}}
\newcommand{\PE}{{\mathrm{PE}}}

\newcommand{\dist}{{\mathrm{dist}}}
\newcommand{\diag}{{\mathrm{diag}}}

\let\phi=\varphi

\let\theta=\vartheta

\newcommand{\DelimArray}[4]{\left#1\begin{array}{*{#3}{c}}#4\end{array}\right#2}
\newcommand{\SDelimArray}[4]{\hbox{\scriptsize\arraycolsep=.5\arraycolsep
  $\left#1\!\!\begin{array}{*{#3}{c}}#4\end{array}\!\!\right#2$}}

\newcommand{\Mat}{\DelimArray()}
\newcommand{\SMat}{\SDelimArray()}

\newenvironment{proof}%[1]%
    {\begin{trivlist} \item {\sl Proof:}} %#1\/:}}%
    {{}\hfill $\square$ \end{trivlist}} %%Hans {}\hfill erg\"{a}nzt und \/ gel\"{o}scht

\newcommand{\tbis}[1]{t_1,t_2,\ldots,t_{#1}}
\newcommand{\tinf}{t_1,t_2,\ldots}
\newcommand{\xbis}[1]{x_1,x_2,\ldots,x_{#1}}
\newcommand{\xinf}{x_1,x_2,\ldots}
\newcommand{\ele}[3]{e_{#1}^{#2}(#3)}
\newcommand{\tele}[3]{\tilde e_{#1}^{#2}(#3)}%tilde-element
\newcommand{\el}[2]{e_{#1}^{#2}}
\newcommand{\tel}[2]{\tilde e_{#1}^{#2}}%tilde-element
\newcommand{\ZX}{\ZZ\langle X\rangle}
\newcommand{\oa}{{\overline\alpha}}
\begin{document}
\title{Jordan homomorphisms and harmonic mappings}
\author{Andrea Blunck \and Hans Havlicek}
\date{}

%\dedication{Version: \today} % deaktivieren

\maketitle

%%%%%%%%%%%%%%%%%%%%%%%%%%%%%%%
\begin{abstract}
We show that each Jordan homomorphism $R\to R'$ of rings gives rise
to a harmonic mapping of one connected component of the projective
line over $R$ into the projective line over $R'$. If there is more
than one connected component then this mapping can be extended in
various ways to a harmonic mapping which is defined on the entire
projective line over $R$.
\end{abstract}

\section{Introduction}\label{se:intro}
 \para
The problem to determine all \emph{harmonic mappings} (see
\ref{para:harmonic}) for \emph{projective lines over rings} (see
\ref{para:proj_gerade}) goes back to \textsc{K.~von~Staudt}
(1798--1867), who treated harmonic bijections between real projective
lines; cf.\ \cite[p.~57--58]{kar+k-88} for a survey and historical
remarks. The results which have been obtained so far show that among
the relevant algebraic mappings are -- apart from projective
transformations -- not only \emph{homomorphisms} but also
\emph{Jordan homomorphisms} of rings (see \ref{para:jordan}). If a
ring $R$ contains a subfield $K$ then the point set of the associated
\emph{chain geometry} (see \ref{para:kettengeom}) is the projective
line over $R$. The investigation of \emph{homomorphism of chain
geometries} (see \ref{para:homo}) has also lead to Jordan
homomorphisms of rings. There is a widespread literature on the
interplay of Jordan homomorphisms, harmonic mappings, and
homomorphisms of chain geometries. The interested reader should
consult \cite{bart-89}, \cite{bart+b-85}, \cite{bart+f-79},
\cite{benz+s+s-81}, \cite{blunck-92b}, \cite{blunck-94},
\cite{cirl+e-90}, \cite{ferr-81}, \cite{herz-87a}, \cite{herz-95},
\cite{hua-53}, \cite{lima+l-77b}, and \cite{lima+l-77a} for further
references and related results.
 \para
Suppose that we are given a Jordan homomorphism $\alpha:R\to R'$ of
rings. The ring $R$ can be embedded in the projective line $\PP(R)$
over $R$ via $t\mapsto R(t,1)$ and there is a similar embedding
$R'\to \PP(R')$. By virtue of these embeddings, $\alpha$ determines a
mapping
\begin{equation}\label{eq:vorspann_affin}
 R(t,1)\mapsto R'(t^\alpha,1')\mbox{ with } t\in R.
\end{equation}
There arises the question if (\ref{eq:vorspann_affin}) can be
extended to a mapping $\PP(R)\to\PP(R')$ in some ``natural way''. It
is fairly obvious that such an extension should take each point
$R(1,t)$ to $R'(1',t^\alpha)$. However, the projective line may also
contain points of the form $R(a,b)$, where neither $a$ nor $b$ are
invertible, whence they cannot be written as $R(1,t)$ or $R(t,1)$.
For each of those points an ``appropriate'' definition of the image
point is not immediate, since $R(a,b)\mapsto R(a^\alpha,b^\alpha)$
gives in general no well defined mapping.
 \par
An affirmative answer to the question above has been given by
\textsc{C.~Bartolone} \cite{bart-89} under the additional assumption
that $R$ is a ring of stable rank $2$. Among the rings with this
property are, e.g., local rings, matrix rings over fields, and
finite-dimensional algebras over commutative fields. In case of
stable rank $2$ each point of $\PP(R)$ can be written in the form
$R(t_1t_2-1,t_1)$ with parameters $t_1,t_2\in R$, and
\begin{equation}\label{eq:vorspann_sr2}
   R(t_1t_2-1,t_1) \mapsto R'(t_1^\alpha t_2^\alpha - 1',t_1^\alpha)
\end{equation}
is a well defined extension of (\ref{eq:vorspann_affin}).
 \par
In the present article there will be no restriction on the rings $R$
and $R'$. We shall show that the mapping (\ref{eq:vorspann_affin})
can be extended to a harmonic mapping $\oa:C\to \PP(R')$, where
$C\subseteq\PP(R)$ denotes the \emph{connected component} (in the
graph-theoretic sense; see \ref{para:proj_gerade}) of the point
$R(1,0)$. If $C\neq\PP(R)$ then $\oa$ can be extended in various ways
to a harmonic mapping $\PP(R)\to\PP(R')$.
\par
The definition of $\oa$ is rather involved: Each point of the
connected component $C$ can be described by some finite sequence
$(\tbis n)$ of parameters in $R$, where $n\geq 0$ is variable. Then
the Jordan homomorphism acts on these parameters, i.e., the sequence
$(t_1^\alpha,t_2^\alpha,\ldots,t_n^\alpha)$ determines the image
point; cf.\ formula (\ref{eq:vorspann_alpha}) below which is a
generalization of (\ref{eq:vorspann_sr2}). The number of parameters
which is needed in order to describe all points of $C$ may be
unbounded. Furthermore, a point of $C$ may admit many representations
in terms of parameters. So the problem is to show that we have a well
defined mapping. In \cite{bart-89} the situation is less complicated:
When $R$ is a ring of stable rank $2$, the projective line $\PP(R)$
coincides with the connected component $C$ and, as has been mentioned
above, each point of $\PP(R)$ can be described with just two
parameters.

\para
The paper is organized as follows: In Section \ref{se:elementary} we
discuss the \emph{elementary subgroup} $\E_2(R)$ of the general
linear group $\GL_2(R)$ over a ring $R$. Following \textsc{P.M.~Cohn}
\cite{cohn-66} we consider a family of matrices $E(t)$, $t\in R$,
with the property that each matrix in $\E_2(R)$ can be written as a
product $E(t_1)E(t_2)\cdots E(t_n)$ with $\tbis n\in R$ and $n\geq
0$. The entries of a matrix in $\E_2(R)$ can be expressed with the
help of an infinite family of polynomials in non-commuting
indeterminates. Next, in Section \ref{se:jordan}, we introduce the
concept of a \emph{polynomial with Jordan property} and present two
infinite families of such polynomials (Propositions \ref{prop:2} and
\ref{prop:3}). These polynomials are used in order to compare the
matrices $E(t_1)E(t_2)\cdots E(t_n)\in\E_2(R)$ and
$E(t_1^\alpha)E(t_2^\alpha)\cdots E(t_n^\alpha)\in\E_2(R')$. For
example, if the $(1,1)$-entry of the first matrix is a unit, then so
is the $(1,1)$-entry of the second matrix (Theorem \ref{thm:inv-0}).
 \par
Unfortunately, in general there is no well defined mapping sending
$E(t_1)E(t_2)\cdots E(t_n)$ to $E(t_1^\alpha)E(t_2^\alpha)\cdots
E(t_n^\alpha)$. But we can pass from $\E_2(R')$ to an appropriate
quotient group $\E_2(R'')/N_\alpha$; here $R''$ denotes the subring
of $R'$ which is generated by the image of the Jordan homomorphism
$\alpha$ and $N_\alpha$ is a normal subgroup of $\E_2(R'')$ which
depends on $\alpha$. In this way a well defined homomorphism of
groups $\E_2(R)\to \E_2(R'')/N_\alpha$ can be obtained (Theorem
\ref{thm:gruppen}). If $N_\alpha$ contains only the identity matrix
then we have a well defined mapping $\E_2(R)\to \E_2(R'')$ (Corollary
\ref{cor:neu}). This is the case when $\alpha$ belongs to a certain
class of Jordan homomorphisms, including homomorphisms and
antihomomorphisms. However, we shall see that $N_\alpha$ can also be
non-trivial (Examples \ref{exa:jordan}). We show that $N_\alpha$ is
in the centre of $\E_2(R'')$, whence the Jordan homomorphism $\alpha$
gives rise to a homomorphism
\begin{equation*}
  \alpha_\PE:\PE_2(R)\to \PE_2(R'')
\end{equation*}
of \emph{projective elementary groups} (see \ref{para:proj_gerade}),
which act on the connected component $C\subseteq\PP(R)$ and a
connected component of the subline $\PP(R'')\subseteq\PP(R')$,
respectively (Theorem \ref{thm:projektiv}). The mapping $\alpha_\PE$
is then the key to showing that a well defined mapping $\oa$ of the
points of $C$ is given by
\begin{equation}\label{eq:vorspann_alpha}
  R(1,0)\cdot E(t_1)E(t_2)\cdots E(t_n)
  \mapsto
  R'(1',0')\cdot E(t_1^\alpha)E(t_2^\alpha)\cdots E(t_n^\alpha)
\end{equation}
with $\tbis n\in R$ and $n\geq 0$. This $\oa$ extends
(\ref{eq:vorspann_affin}) and it turns $\alpha_\PE$ into a
homomorphism of transformation groups. We show some geometric
properties of the mapping $\oa$ and present several examples to
illustrate our results.
\par
In Section \ref{se:homomorph} we examine homomorphisms of chain
geometries. In particular, it will be established that the
isomorphisms of affine chain geometries discussed by
\textsc{A.~Herzer} in \cite[9.1]{herz-95} can be extended to
homomorphisms of chain geometries without any additional assumption
on the underlying rings (Theorem \ref{thm:ketten}). Thus our results
yield new examples of homomorphisms of chain geometries.

\section{The elementary subgroup $\E_2(R)$}\label{se:elementary}

\para
Throughout this paper we shall only consider associative rings with a
unit element, which is preserved by homomorphisms, inherited by
subrings, and acts unitally on modules.  The group of
\emph{invertible elements} and the \emph{centre} of a ring $R$ will
be denoted by $R^*$ and $Z(R)$, respectively. Also, we shall write
$\cS(R):=R^0\cup R^1\cup R^2\cup\ldots$ for the set of all
\emph{finite sequences} in $R$, including the empty sequence.
\para
We recall that the \emph{elementary subgroup\/} $\E_2(R)$ of the
general linear group $\GL_2(R)$ is generated by the set of all
matrices
\begin{equation}\label{def:E}
  E(t):=\Mat2{t&1\\-1&0} \mbox{ with } t\in R.
\end{equation}
Furthermore,
\begin{equation}\label{eq:E-inv}
  E(t)^{-1}=\Mat2{0&-1\\1&t}=E(0)\cdot E(-t)\cdot E(0),
\end{equation}
whence each element of $\E_2(R)$ can be written in the form
\begin{equation}\label{def:E(T)}
  E(t_1)\cdot E(t_2)\cdots E(t_n)=:E(\tbis n)=:E(T)
\end{equation}
where ${T}:=(\tbis n)\in\cS(R)$ denotes a sequence of $n\geq 0$
elements; cf.\ \cite[p.~368]{cohn-66}.
\par
It is easily seen that a ($2\times 2$)-matrix over $R$ commutes with
all matrices $E(t)$, $t\in R$, if and only if it has the form
$\diag(a,a)$ with $a\in Z(R)$. Hence the centre of $\E_2(R)$ is the
subgroup
\begin{equation}\label{def:zentrum}
  H:=\E_2(R)\cap \{\diag(a,a)\mid a\in Z(R)^*\}.
\end{equation}
 \para
In order to describe the entries of a matrix (\ref{def:E(T)}) we
consider an infinite sequence $X=(\xinf)$ of indeterminates over
$\ZZ$ and the free $\ZZ$-algebra $\ZX$. Its elements are polynomials
in the non-commuting indeterminates $\xinf$ with coefficients in
$\ZZ$.
\par
We shall frequently use the following \emph{universal property} of
$\ZX$ \cite[p.~6]{lam-91}: If $R$ is an arbitrary ring and $(\tinf)$
is an infinite sequence of elements in $R$, then there is a unique
homomorphism $\ZX\to R$ such that $x_i\mapsto t_i$ for each
$i\in\{1,2,\ldots\}$. The image of $f\in\ZX$ under this homomorphism
is written as $f(\tinf)$. Also, we have a homomorphism
$\E_2(\ZX)\to\E_2(R)$ by the action of $f\mapsto f(\tinf)$ on the
entries of a matrix. In addition, let $T=(\tbis n)\in\cS(R)$ be a
finite sequence which may be empty ($n<1$). Then we put
\begin{equation}\label{def:subst}
  f(T)=f(\tbis n):=f(\tbis n,0,0,\ldots).
\end{equation}
 \par
In order to avoid misinterpretations let us point out the following
particular case of (\ref{def:subst}): Assume that
$f=2x_2+x_3\in\ZX$, $T=(x_2,x_3)\in\cS(\ZX)$, and $V=(v_1,v_2,v_3)\in
\cS(R)$. Then $f(T)=f(x_2,x_3)$ denotes that polynomial which arises
from $f\in\ZX$ if $X$ is substituted by $(x_2,x_3,0,0,\ldots)$. As
$f(x_2,x_3)=2x_3\neq f$, we must not write ``$f=f(x_2,x_3)$'' in
order to stress that $f$ belongs to the $\ZZ$-subalgebra of $\ZX$
generated by $\{x_2,x_3\}$. Furthermore, $f(V)=2v_2+v_3$, but
$(f(T))(V)=(f(x_2,x_3))(V)=2x_3(V)=2v_3=f(v_2,v_3)$.
 \par
On the other hand, for each $g\in\ZX$ there is a sufficiently large
integer $n$ such that $g=g(\xbis n)$.
\para
Following \cite[p.~376]{cohn-66}, we define a sequence of elements in
$\ZX$ recursively by
  \begin{equation}\label{def:e}
  \renewcommand{\arraystretch}{1.3}%% wirkt nur lokal in equation-Umgebung
  \left.\begin{array}{c}
   \el {}{(-2)}  :=  -1,\;\;
   \el {}{(-1)}  :=  0,\;\;
   \el {}{(0)}   :=  1,\\
   \el {}{(n)}   := \el{}{(n-1)}\,x_n - \el{}{(n-2)},
  \end{array}\right\}
  \end{equation}
where $n\in\{1,2,\ldots\}$. It will turn out useful to have a short
notation for polynomials that arise from the ones given in
(\ref{def:e}) by a substitution (\ref{def:subst}) as follows: Given
$i,j\in\ZZ$ with $i\geq 1$ and $j\geq i-3$ we define
%%neu
\begin{equation}\label{def:e-ij}
   \el i j := \ele  {} {(j-i+1)}{x_i,x_{i+1},\ldots,x_{j}},\;\;
   \tel i j := \ele  {} {(j-i+1)}{x_j,x_{j-1},\ldots,x_{i}}.
\end{equation}
So $j$ is an upper index rather than an exponent. In particular, we
have
\begin{equation}\label{eq:e(n)}
\el{1}{n}=\ele{}{(n)}{\xbis n}=  \el{}{(n)} \mbox{ for all }
n\in\{-2,-1,\ldots\}.
\end{equation}
Furthermore, each polynomial $\el ij$ can be written as a $\ZZ$-linear
combination of monomials $x_{h_1}x_{h_2}\cdots x_{h_m}$ with $i\leq
h_1<h_2<\cdots<h_m\leq j$ and $m$ ranging from $0$ to $j-i+1$.
Likewise $\tel ij$ is a $\ZZ$-linear combination of monomials
$x_{h_1}x_{h_2}\cdots x_{h_m}$ with $j\geq h_1>h_2>\cdots>h_m\geq i$
and $m$ ranging from $0$ to $j-i+1$. For example, $\el 2 2
=\ele{}{(1)}{x_2}= x_2$ and $\el 5 6 = \ele{}{(2)}{x_5,x_6}=x_5 x_6
-1$.
 \par
Many of the following calculations are based upon the identities
\begin{eqnarray}
 \label{eq:e-ij-rechts}
 \el{i}{j}  &=& \el{i}{j-1}x_j - \el{i}{j-2},
 \\
 \tel{i}{j} &=& \tel{i+1}{j}x_i - \tel{i+2}{j},
 \label{eq:te-ij-rechts}
\end{eqnarray}
which follow from (\ref{def:e}) and (\ref{def:e-ij}) whenever $1\geq
i\geq j$.
 \par
Next we describe certain elements of the group $\E_2(\ZX)$:

\begin{lem}\label{lem:1}
If $(\xbis n)\in \cS(\ZX)$ then
%% neu: If statt given
\begin{eqnarray}\label{eq:Ex}
  E(\xbis n) &=& \Mat2{ \el1{n}    &  \el1{n-1} \\
                       -\el2{n}    & -\el2{n-1}},
 \\
 \label{eq:Ex-inv}
 E(\xbis n)^{-1} &=& \Mat2{-\tel 2{n-1} & -\tel 1{n-1} \\
                        \tel 2{n  } &  \tel 1{n  } }.
\end{eqnarray}
\end{lem}

\begin{proof}
Clearly, for $n=0$ we have $E()=I$, the identity in $\E_2(\ZX)$. Now
(\ref{eq:Ex}) follows easily by induction (\cite[p.~376]{cohn-66}),
since for $n\geq 1$ we infer from the induction hypothesis and
(\ref{eq:e-ij-rechts}) that
\begin{equation*}
  E(\xbis {n}) =
  \Mat2{ \el1{n-1}  &  \el1 {n-2} \\
        -\el2{n-1}  & -\el2 {n-2}}
        \cdot
        \Mat2{x_{n}&1\\-1&0}
        =
        \Mat2{ \el1{n}    &  \el1{n-1} \\
              -\el2{n}    & -\el2{n-1}}.
\end{equation*}
The proof of (\ref{eq:Ex-inv}) runs in a similar manner taking into
account the first part of equation (\ref{eq:E-inv}),
(\ref{eq:te-ij-rechts}), and $E(\xbis n)^{-1} =
E(x_2,x_3,\ldots,x_{n})^{-1} \cdot E(x_1)^{-1}$.
\end{proof}
   \para
We obtain the recursion $\el{1}{n} = x_1 \el2 n - \el3 n$ for
$n\in\{1,2,\ldots\}$ from the $(1,1)$-entry of the matrix equation $E(\xbis
n)=E(x_1)\cdot E(x_2,x_3,\ldots,x_n)$ together with  (\ref{eq:Ex}). This yields
\begin{eqnarray}
  \label{eq:e-ij-links}
  \el{i}{j} &=& x_i \el{i+1} j - \el{i+2} j,
  \\
  \label{eq:te-ij-links}
  \tel{i}{j} &=& x_j \tel{i} {j-1} - \tel{i} {j-2},
\end{eqnarray}
for $1\geq i\geq j$ as counterparts of (\ref{eq:e-ij-rechts}) and
(\ref{eq:te-ij-rechts}); cf.\ \cite[p.~376]{cohn-66}.

 \para
We now return to an arbitrary ring $R$. For each $T=(\tbis n)
\in\cS(R)$ there is the homomorphism $\ZX\to R:f\to f(T)$; see
(\ref{def:subst}). So we can transfer our calculations from $\ZX$
to $R$ and from $\E_2(\ZX)$ to $\E_2(R)$. For example,
(\ref{eq:e-ij-rechts}) yields
\begin{equation*}
  \ele{1}{n}T
  =
  \ele{1}{n-1}{T} t_n - \ele{1}{n-2}{T}
  = \ele{1}{n-1}{\tbis{n-1}} t_n - \ele{1}{n-2}{\tbis{n-2}}.
\end{equation*}
whereas (\ref{eq:e-ij-links}) gives
\begin{equation*}
  \ele{1}{n}T
  =
  t_1\ele{2}{n}{T} - \ele{3}{n}{T}
  = t_1\ele{1}{n-1}{t_2,t_3,\ldots,t_n} - \ele{1}{n-2}{t_3,t_4,\ldots,t_n}.
\end{equation*}
Observe that there are numerous ways to rewrite such identities,
since we may also add irrelevant ring elements. E.g.,
$\ele{1}{n}T=\ele{2}{n+1}{s,T}$ and $\ele{1}{n}T=\ele{1}{n}{T,v}$ for
all $s,v\in R$.
 \par
It has been pointed out in \cite[Lemma 1.5]{bart-89} that
$\ele{1}{3}{t_1,t_2,t_3}=t_1 t_2 t_3-t_3-t_1\in R^*$ implies
$\tele{1}{3}{t_1,t_2,t_3}=t_3 t_2 t_1-t_1-t_3\in R^*$ for all
$(t_1,t_2,t_3)\in R^3$. The following result is a generalization of
this, but we give a completely different proof using the group
$\E_2(R)$:

\begin{prop}\label{prop:1}
Let ${T}\in R^n$, $n\geq 0$. Then $\ele1nT\in R^*$ implies $\tele1nT
\in R^*$.
\end{prop}

\begin{proof}
This is obvious for $n=0$. So let $n\geq 1$. Since $\ele1{n}T\in
R^*$,
%%neu Since statt from
there are elements $s,\,v\in R$ such that
\begin{eqnarray*}
 0 &=& s\ele1{n}{T} -\ele2{n}{T}
    =  s\ele2{n+1}{s,T} -\ele3{n+1}{s,T}
    =  \ele1{n+1}{s,{T}}
    =  \ele1{n+1}{s,{T},v},
       \\
 0&=& \ele1{n}{T}v-\ele1{n-1}{T}
   =  \ele1{n}{T,v}v-\ele1{n-1}{T,v}
   =  \ele1{n+1}{T,v}
   =  \ele2{n+2}{s,T,v},
\end{eqnarray*}
where we used (\ref{eq:e-ij-links}) and (\ref{eq:e-ij-rechts}). We
read off from (\ref{eq:Ex}) and $\ele2{n+1}{s,T,v}=\ele1{n}{T}$ that
\begin{equation*}
  E(s,{T},v) =
  \diag\big(\ele1{n+2}{s,T,v},-\ele1{n}{T}\big).
\end{equation*}
The inverse of this matrix is diagonal, too, and by (\ref{eq:Ex-inv})
its $(1,1)$-entry equals $-\tele2{n+1}{s,T,v}=-\tele1{n}T$, which is
therefore a unit.
\end{proof}

\section{Jordan homomorphisms}\label{se:jordan}
\para\label{para:jordan}
%%NEU!!
Let $R$ and $R'$ be rings. We recall that a mapping $\alpha:R\to R'$
is called \emph{Jordan homomorphism} if
\begin{equation}\label{def:jordan}
  (a+b)^\alpha = a^\alpha + b^\alpha,\;\;
  1^\alpha = 1',\;\;
  (aba)^\alpha = a^\alpha b^\alpha a^\alpha\;\;\;
  \mbox{for all } a,b\in R.
\end{equation}
See, among others, \cite[p.~2]{jac-68} or \cite[p.~832]{herz-95}.
%% neu For statt Given
For such an $\alpha$ and an $a\in R^*$ the equation
$1'=(aa^{-2}a)^\alpha = a^\alpha (a^{-2})^\alpha a^\alpha$ shows that
$a^\alpha\in {R'}^*$. Also, $a^\alpha=(aa^{-1}a)^\alpha=a^\alpha
(a^{-1})^\alpha a^\alpha$
%%neu-from
implies $(a^{-1})^\alpha = (a^\alpha)^{-1}$ \mbox for all $a\in R^*$.
We say that $\alpha$ is \emph{proper} if it is neither a homomorphism
nor an antihomomorphism.
 \par
In general a Jordan homomorphism is not multiplicative, but on
certain expressions, like $aba$ in (\ref{def:jordan}), it acts ``like
a ring homomorphism''. In order to generalize this we pass to the
%%neu statt go over
free $\ZZ$-algebra $\ZX$. We say that $f\in\ZX$ is a \emph{polynomial
with Jordan property}, or a \emph{J-polynomial} for short, if
$f(T)^\alpha = f(T^\alpha)$ holds
%%neu
for every Jordan
homomorphism $\alpha:R\to R'$ between arbitrary rings $R$ and $R'$  and every (finite or infinite) sequence
$T$ in $R$.

\begin{exas}\label{exa:j-poly}
Clearly, $1$, $x_i$, and $x_ix_jx_i$ are J-polynomials for all
$i,j\in\{1,2,\ldots\}$. The set of all J-polynomials forms a
$\ZZ$-submodule of $\ZX$. Furthermore, suppose that $G$ is a (finite or
infinite) sequence of J-polynomials and that $f$ is a J-polynomial.
Then it is easily seen that also $f(G)$ has the Jordan property. In
particular we obtain the J-polynomials
\begin{eqnarray*}
                 x_1^2 &=& x_1 1 x_1,\\
         x_1x_2+x_2x_1 &=& (x_1+x_2)^2-x_1^2-x_2^2,\\
   x_1x_2x_3+x_3x_2x_1 &=& (x_1+x_3)x_2(x_1+x_3)-x_1x_2x_1-x_3x_2x_3.
\end{eqnarray*}
\end{exas}

Next we give two infinite families of J-polynomials.

\begin{prop}\label{prop:2}
Let $n\in\{-1,0,\ldots\}$. Then $\el 1n \tel1{n-1}$ is a polynomial
with Jordan property.
\end{prop}
\begin{proof}
We proceed by induction observing that $\el 1{-1} \tel1{-2}=0$, $\el
10 \tel1{-1}=0$, and $\el 1{1} \tel1{0}=x_1$ are J-polynomials.
Letting $n\geq 2$ we infer from the induction hypothesis and an
appropriate substitution that
\begin{equation}\label{eq:ind-vs}
 \el{i}{n}\tel{i}{n-1} =
  \ele{1}{n-i+1}{x_i,x_{i+1},\ldots,x_{n}  }
 \tele{1}{n-i}  {x_i,x_{i+1},\ldots,x_{n-1}}
\end{equation}
is a J-polynomial for $i\in\{2,3\}$. The keys for the following
calculations are formula (\ref{eq:e-ij-links}) and formula
(\ref{eq:te-ij-rechts}). So we get
\begin{eqnarray}
   \el1n\tel1{n-1}
   &=& (x_1\el2n - \el3n)(\tel2{n-1}x_1 - \tel3{n-1})
  \nonumber
  \\
  &=&
   \underbrace{
   x_1 (\el2n\tel2{n-1}) x_1 + \el3n\tel3{n-1}
              }_{=:f_1}
  -(x_1\el2n\tel3{n-1} + \el3n\tel2{n-1} x_1)
  \nonumber
\end{eqnarray}
where, by (\ref{eq:ind-vs}) and the examples given in
\ref{exa:j-poly}, $f_1$ is a J-polynomial. Similarly, if $n\geq 3$
then
\begin{equation*}
  x_1 \el2n\tel3{n-1} + \el3n\tel2{n-1} x_1
  =
   \underbrace{
   x_1 x_2 (\el3n\tel3{n-1}) + (\el3n\tel3{n-1}) x_2 x_1
              }_{=:f_2}
  -(x_1\el4n\tel3{n-1} + \el3n\tel4{n-1} x_1)
\end{equation*}
with a J-polynomial $f_2$. Also, for $n\geq 4$ we get
\begin{equation*}
  x_1\el4n\tel3{n-1} + \el3n\tel4{n-1}  x_1
  =    \underbrace{
    x_1 (\el4n\tel4{n-1})x_3 + x_3(\el4n\tel4{n-1}) x_1
                  }_{=:f_3}
  -(x_1\el4n\tel5{n-1} +\el5n\tel4{n-1}x_1).
\end{equation*}
with a J-polynomial $f_3$. Proceeding in this way we arrive either at
\begin{equation*}
f_{n-1} - (x_1 \el{n+1}{n}\tel{n}{n-1} + \el{n}n\tel{n+1}{n-1} x_1 )
=
f_{n-1} - (x_1\cdot 1\cdot 1 + x_n\cdot 0\cdot x_1),
\end{equation*}
when $n$ is odd, or at
\begin{equation*}
f_{n-1} - (x_1 \el{n}n \tel{n+1}{n-1} + \el{n+1}n \tel{n}{n-1} x_1)
=
f_{n-1} - (x_1\cdot x_n\cdot 0 + 1\cdot 1\cdot x_1),
\end{equation*}
when $n$ is even. But $x_1$ is a J-polynomial, whence the assertion
follows.
\end{proof}

\begin{prop}\label{prop:3}
Let $n\in\{-2,-1,\ldots\}$. Then $\el 1n \tel1n$ is a polynomial with
Jordan property.
\end{prop}
\begin{proof}
This is clear when $n<0$. If $n\geq 0$ then we infer from the
$(1,2)$-entry of $E(\xbis n)\cdot E(\xbis n)^{-1}=I$ and Lemma
\ref{lem:1} that
\begin{equation*}
  -\el1 n \tel 1 {n-1} + \el1{n-1} \tel1{n} = 0.
\end{equation*}
Hence $\el 1 {n+1} \tel 1 n
  = \el 1 n x_{n+1} \tel 1n - \el 1{n-1}\tel 1n
  = \el 1 n x_{n+1} \tel 1n - \el 1{n}  \tel 1{n-1}$,
where we used (\ref{eq:e-ij-rechts}) and the equation above. So
\begin{equation*}
  \el 1 n x_{n+1} \tel 1n = \el 1 {n+1} \tel 1 n + \el 1{n} \tel 1{n-1},
\end{equation*}
and Proposition \ref{prop:2} yields that $\el 1 n x_{n+1} \tel 1n$ is
a J-polynomial. This property remains unaltered if we substitute
$x_{n+1}$ by $1$.
\end{proof}

Our next result generalizes \cite[Lemma 1.2]{lima+l-77a} and is based
upon the previous propositions. It will be the backbone of many
considerations. The theorem says that if certain entries of a matrix
$E(T)$ are of a particular form, then so are the corresponding
entries in $E(T^\alpha)$.

\begin{thm}\label{thm:inv-0}
Let $\alpha:R\to R'$ be a Jordan homomorphism and let $T\in R^n$,
$n\geq 0$. Then the following holds:
\begin{enumerate}
  \item $\ele 1nT\in R^*$ implies $\ele 1n{T^\alpha}\in {R'}^*$.
  \item If $\ele 1nT\in R^*$ and $\ele 1{n-1}{T}=0$ then
   $\ele 1{n-1}{T^\alpha}=0'$.
\end{enumerate}
\end{thm}
\begin{proof}
(a) We deduce from Proposition \ref{prop:1} that $\tele 1nT\in R^*$.
Proposition \ref{prop:3} and $R^{*\alpha}\subseteq {R'}^*$ establish
that
\begin{equation*}
  \ele 1n{T^\alpha}\tele 1n{T^\alpha}
  = (\el 1n\tel 1n)(T^\alpha)
  =\big((\el 1n \tel 1n)(T)\big)^\alpha
  =\big(\ele 1nT \tele 1nT\big)^\alpha \in{R'}^*.
\end{equation*}
So $\ele 1n{T^\alpha}$ is right invertible. Let $\widetilde T$ be the
finite sequence $T$ written in reverse order. Then
\begin{equation*}
 \tele 1n{T^\alpha}\ele 1n{T^\alpha}
 =\ele 1n{\widetilde T^\alpha}\tele 1n{\widetilde T^\alpha}
 =\big(\ele 1n{\widetilde T} \tele 1n{\widetilde T}\big)^\alpha
 =\big(\tele 1n{ T}           \ele 1n{T}\big)^\alpha
 \in {R'}^*,
\end{equation*}
where we used the equation from above in the second step. Hence $\ele
1n{T^\alpha}$ is also left invertible.
\par
(b) We read off from (\ref{eq:e-ij-rechts}) the identity
\begin{equation}\label{eq:e.1}
\ele{1}{n+1}{\xbis n,1} = \el{1}{n}\cdot 1 - \el 1{n-1}.
\end{equation}
So $\ele 1{n-1}{T}=0$ implies $\ele{1}{n+1}{T,1} = \ele{1}{n}{T}$,
whereas $\tele{1}{n}{T,1}=\tele{1}{n}{T}$ holds trivially. Since
$\el{1}{n+1}\tel{1}{n}$ and $\el{1}{n}\tel{1}{n}$ are J-polynomials
by Proposition \ref{prop:2} and Proposition \ref{prop:3}, we obtain
that
\begin{eqnarray*}
    \ele{1}{n+1}{T^\alpha,1'}\tele{1}{n}{T^\alpha,1'}
  &=& (\el{1}{n+1}\tel{1}{n})(T^\alpha,1')
  = \big((\el{1}{n+1}\tel{1}{n})(T,1)\big)^\alpha\\
  &=& \big(\ele{1}{n}{T}\tele{1}{n}{T}\big)^\alpha
  = \ele{1}{n  }{T^\alpha}   \tele{1}{n}{T^\alpha}.
\end{eqnarray*}
Now Proposition \ref{prop:1} and (a) allow to cancel the unit
$\tele{1}{n}{T^\alpha,1'}=\tele{1}{n}{T^\alpha}$. Hence
(\ref{eq:e.1}) forces that $\ele 1{n-1}{T^\alpha} = 0'$.
\end{proof}

We are now in a position to show our first main result.

\begin{thm}\label{thm:gruppen}
Let $\alpha:R\to R'$ be a Jordan homomorphism and denote by $R''$ the
subring of $R'$ generated by $R^\alpha$. Then the following
statements are true:
\begin{enumerate}
\item
If $T\in\cS(R)$ then $E(T)\in H$ implies $E(T^\alpha)\in H''$,
%%neu If statt Given
where $H$ and $H''$ denote the centres of $\E_2(R)$ and $\E_2(R'')$,
respectively.
\item
The set
 \begin{equation}\label{def:N-alpha}
  N_\alpha:=\{E(T^\alpha)\mid T\in\cS(R) \mbox{ and } E(T)=I\}
 \end{equation}
is contained in $H''$.
\item
$N_\alpha$ is a normal subgroup of\/ $\E_2(R'')$. Furthermore, the
mapping
  \begin{equation}\label{def:gruppenabb}
    \alpha_\E : \E_2(R) \to \E_2(R'')/N_\alpha :
    E(T)\mapsto N_\alpha\cdot E(T^\alpha),
  \end{equation}
where $T\in\cS(R)$, is a well defined homomorphism of groups.
\end{enumerate}
\end{thm}

\begin{proof}
(a) Let $T\in R^n$ be a sequence such that $E(T)\in H$. So, by
%%neu-from
(\ref{def:zentrum}), there is an $a\in Z(R)^*$ with $E(T)
=\diag(a,a)$. Put $\SMat2{a' & b'\\c' &d'}:=E(T^\alpha)$. By virtue
%%neu-from
of (\ref{eq:Ex}), all entries of this matrix are in $R''$. The first
row of $E(T)$ is $\big(\ele{1}{n}{T},\ele{1}{n-1}{T}\big)=(a,0)$. We
infer from Theorem \ref{thm:inv-0} that $a'$ is a unit and that
$b'=0'$. Next choose any $s\in R$ and let $S:=(s,T,0,-s,0)$. By
(\ref{eq:E-inv}),
\begin{eqnarray}%%neu 0'
  E(S)
  &=&
  E(s)E(T)E(s)^{-1} = E(T) = \diag(a,a),\\
  E(S^\alpha)
  &=&
  E(s^\alpha)E(T^\alpha)E(s^\alpha)^{-1}
  =\Mat2{d' & d's^\alpha-s^\alpha a'-c'\\
         0' & a'                         }.
\end{eqnarray}
As before, Theorem \ref{thm:inv-0} shows that $d'$ is a unit and that
$d's^\alpha-s^\alpha a'-c'$ vanishes for all $s\in R$. Letting
$s:=0,1$ yields $c'=0'$ and $a'=d'$. Hence $a's^\alpha=s^\alpha a'$
for all $s\in R$ and therefore $a'\in Z(R'')$. We infer that
$E(T^\alpha)\in H''$.
 \par
(b) This is immediate from (a).
 \par
(c) First, let us introduce the following notation. Given $T\in R^n$,
$n\geq 0$, we define
%%neu wg. given
  \begin{equation}\label{def:T-dach}
   \widehat T := (0,-t_n,0,0,-t_{n-1},0,\ldots,0,-t_1,0)\in R^{3n}.
  \end{equation}
We observe that (\ref{eq:E-inv}) implies $E(\widehat T)=E(T)^{-1}$.
Also, when (\ref{def:T-dach}) is applied accordingly to sequences in
$R'$, then $\widehat T^\alpha=\widehat{T^\alpha}$.
\par
Clearly, $I'\in N_\alpha$. Next assume that $E(T)=E(V)$ for
$T,V\in\cS(R)$. Then $E(T,\widehat V)=I$, and $
E(T^\alpha)E(V^\alpha)^{-1} = E(T^\alpha,\widehat{V^\alpha}) =
E(T^\alpha,\widehat{V}^\alpha) \in N_\alpha$. Hence $E(T^\alpha)\in
N_\alpha \cdot E(V^\alpha)$.
\par
Any two matrices in $N_\alpha$ can be written in the form
$E(T^\alpha)$, $E(V^\alpha)$ with $E(T)=E(V)=I$. By the above,
$E(T^\alpha)E(V^\alpha)^{-1}\in N_\alpha$. So $N_\alpha$ is a
subgroup of $\E_2(R'')$ which is normal by (a).
\par
Summing up, the mapping $\alpha_\E$ is well defined and obviously it
is a homomorphism.
\end{proof}

The reason for introducing the subring $R''$ in the theorem above is
that $N_\alpha$ need not be normal in $\E_2(R')$; cf.\ Example
\ref{exa:jordan}~(f) below.

\begin{cor}\label{cor:neu}%%neu
If, under the assumptions of Theorem \ref{thm:gruppen}, the group
$N_\alpha$ is trivial, then
\begin{equation*}
   \alpha_\E : \E_2(R) \to \E_2(R'') : E(T) \mapsto E(T^\alpha),
\end{equation*}
where $T\in\cS(R)$, is a well defined mapping.
\end{cor}

\begin{exas}\label{exa:jordan}
\begin{enumerate}
\item
Let $\alpha:R\to R'$ be a homomorphism. Then there is the
homomorphism $\alpha_*:\GL_2(R)\to \GL_2(R') : M\mapsto M^\alpha$,
i.e., $\alpha$ is applied to each entry of $M$. Obviously,
$E(t)^{\alpha_*}=E(t^\alpha)$ for all $t\in R$, whence
$N_\alpha=\{I'\}$ and $\alpha_*$ restricts to $\alpha_\E:\E_2(R)\to
\E_2(R'')$.
\item
Let $\alpha:R\to R'$ be an antihomomorphism. The mapping
$\alpha_{**}:\GL_2(R)\to \GL_2(R') : M\mapsto
E(0')^{-1}((M^\T)^{\alpha_*})^{-1} E(0')$, where $M^\T$ denotes the
transpose of $M$, is a homomorphism. It maps each $E(t)\in\E_2(R)$ to
$E(t^\alpha)$. Hence $N_\alpha=\{I'\}$ and $\alpha_{**}$ restricts to
$\alpha_\E$. If $R'$ is commutative then $M^{\alpha_{**}}=
\det(M^{\alpha_*})^{-1}M^{\alpha_*}$ for all $M\in\GL_2(R)$.
\item
Suppose that $R=\prod_{\lambda\in\Lambda}R_\lambda$ is the direct
product of rings $R_\lambda$. Then, up to isomorphism,
$\GL_2(R)=\prod_{\lambda\in\Lambda}\GL_2(R_\lambda)$ and
$\E_2(R)=\prod_{\lambda\in\Lambda}\E_2(R_\lambda)$. Similarly, let
$R'=\prod_{\lambda\in\Lambda}R'_\lambda$ and let $\alpha_\lambda
: R_\lambda \to R'_\lambda$ be a family of mappings, where each
$\alpha_\lambda$ is a homomorphism or antihomomorphism. Then
$\alpha:=\prod_{\lambda\in\Lambda} \alpha_\lambda$ is in general a
proper Jordan homomorphism $R \to R'$. Now, by (a) or (b), we can
choose at least one homomorphism $\beta_\lambda:\GL_2(R_\lambda)\to
\GL_2(R'_\lambda)$ which restricts to $\alpha_{\lambda,\E}$. So
$\beta:=\prod_{\lambda\in\Lambda}\beta_\lambda : \GL_2(R)\to
\GL_2(R')$ is a homomorphism with $E(t)^\beta=E(t^\alpha)$ for all
$t\in R$. Finally, as above, $N_\alpha=\{I'\}$ and $\beta$ restricts
to $\alpha_\E$.
\item
The following class of examples is essentially due to
\textsc{A.~Herzer} \cite[4.2]{herz-87a}. Although our assumptions are
weaker, \textsc{Herzer's} proofs immediately carry over to our
setting.
\par
Suppose that $D$ is a commutative ring and let $B$ be a $D$-algebra
with a $D$-linear homomorphism $\chi:B\to D$ of rings. Denote by $M$
a left module over $D$ which is endowed with a $D$-bilinear,
alternating, and associative product $M\times M\to M$. Then
$R:=B\oplus M$ becomes a $D$-algebra, if a product is defined by
\begin{equation*}
  (b_1+m_1)(b_2+m_2):=b_1b_2 + {b_1}^\chi m_2 + {b_2}^\chi m_1
  + m_1m_2
\end{equation*}
for all $b_1,b_2\in B$ and $m_1,m_2\in M$. (The commutativity of $D$
guarantees that $R$ is associative.) Assume that $B'$, $\chi'$, and
$M'$ are given as above and that $\alpha_1:B\to B'$ is a homomorphism
or antihomomorphism of $D$-algebras satisfying $\chi=\alpha_1\chi'$.
Also, let $\alpha_2:M\to M'$ be an arbitrary $D$-linear mapping. Then
\begin{equation*}
  \alpha : R \to R' : b+m \mapsto b^{\alpha_1}+m^{\alpha_2}\;\;(b\in
  B,\; m\in M)
\end{equation*}
is a $D$-linear Jordan homomorphism.
\item
With the notation introduced in Example (d) let $B=B'$, $\chi=\chi'$,
and $M=M'=D^3$. Write $(\eps_1,\eps_2,\eps_3)$ for the canonical
basis of $D^3$. A product on $D^3$ with the properties mentioned
above is given by $\eps_i^2=0$, $\eps_i\eps_j=-\eps_j\eps_i$,
$\eps_1\eps_2=\eps_3$, $\eps_i\eps_3=0$ for all $i,j\in\{1,2,3\}$. We
define a $D$-linear Jordan automorphism $\alpha$ of $R$ by
$\alpha_1=\id_B$, whereas $\alpha_2$ fixes $\eps_1$ and interchanges
$\eps_2$ with $\eps_3$. Now $N_\alpha\neq\{I\}$ follows from
\begin{eqnarray*}
  E(\eps_1)E(\eps_3)E(-\eps_1)E(-\eps_3)&=& I,\\
  E(\eps_1)E(\eps_2)E(-\eps_1)E(-\eps_2)&=& \diag(1-\eps_3,1-\eps_3).
\end{eqnarray*}
So $\alpha$ is a proper Jordan automorphism. Also there is no mapping
$\E_2(R)\to \E_2(R)$ with $E(T)\mapsto E(T^\alpha)$ for all
$T\in\cS(R)$.
\item
Let $R$ be given as in Example (e) with $B=D$ and $\chi=\id_D$. (Then
$R=D^4$ is isomorphic to the exterior algebra $\bigwedge D^2$.)
Furthermore $\alpha:R\to R$ is defined as in (e), but we reserve the
letter $R'$ for the $D$-algebra of ($4\times 4$)-matrices over $D$.
The right regular representation $\rho$ of $R$ maps each $a\in R$ to
that ($4\times 4$)-matrix which describes the linear mapping $R\to R:
x\mapsto xa$ in terms of the basis $(1,\eps_1,\eps_2,\eps_3)$. The
product $\alpha\rho$ is a Jordan monomorphism, say $\beta:R\to R'$.
We have $R''=R^\beta=R^\rho$.
\par
The first row of the ($4\times 4$)-matrix $u':=(1-\eps_3)^\rho\in R'$
reads $(1,0,0,-1)$, since $1(1-\eps_3)=1-\eps_3$. So it is not in the
centre of $R'$ and there is a ($4\times 4$)-matrix $r'\in R'$ that
does not commute with $u'$. A straightforward calculation shows that
$E(r')^{-1}\diag(u',u')E(r')$ is not diagonal. On the other hand, all
matrices in $N_\beta$ are diagonal, since they are in the centre
$H''$ of $\E_2(R'')$.  Example (e) shows that  $\diag(u',u')\in N_\beta$,
whence $N_\beta$ is not normal in $\E_2(R')$.
\item
With $R$ as in Example (e) let $\alpha_1=\id_B$, whereas $\alpha_2$
is given by  $\eps_1\mapsto\eps_1$, $\eps_2\mapsto\eps_2$, and
$\eps_3\mapsto 0$. Then $R^\alpha$ is not a subring of $R$, since
$\eps_3=\eps_1^\alpha\eps_2^\alpha$ is not in the image of $\alpha$.
\end{enumerate}
\end{exas}

\section{The projective line over a ring}\label{se:line}
\para\label{para:proj_gerade}%%NEU!!
Consider the free left $R$-module $R^2$. The \emph{projective line
over\/} $R$ is the orbit of the free cyclic submodule $R(1,0)$ under
the natural right action of $\GL_2(R)$. In other words, $\PP(R)$ is
the set of all $p\leq R^2$ such that $p=R(a,b)$, where $(a,b)$ is the
first row of an invertible matrix; compare \cite[p.~785]{herz-95}. If
also $(c,d)$ is the first row of an invertible matrix,  then
$R(a,b)=R(c,d)$ if and only if there is a unit $u\in R^*$ with
$(c,d)=u(a,b)$ \cite[Proposition 2.1]{blu+h-00b}.
 \par
Let $\{(a,b), (c,d)\}$ be a basis of $R^2$. Then the points
$p=R(a,b)$ and $q= R(c,d)$ are called \emph{distant}. In this case we
write $p\,\dis\,q$. The vertices of the \emph{distant graph} on
$\PP(R)$ are the points of $\PP(R)$, the edges of this graph are the
unordered pairs of distant points. The set $\PP(R)$ can be decomposed
into \emph{connected components} (maximal connected subsets of the
distant graph); cf.\ \cite[p.~108]{blu+h-01a}.
\par
The orbit of $R(1,0)\in\PP(R)$ under the group $\E_2(R)$ is exactly
the connected component of $R(1,0)$ \cite[Theorem 3.2]{blu+h-01a}. It
will be denoted by $C$. If a matrix $E(T)\in\E_2(R)$ fixes $R(1,0)$
and all points $R(t,1)$ with $t\in R$, then it is easily seen that
$E(T)=\diag(a,a)$ with $a\in Z(R)^*$. On the other hand, each such
$E(T)$ fixes all points of $\PP(R)$. So the kernel of the group
action of $\E_2(R)$ on $C$ (or $\PP(R)$) is the centre $H$ of
$\E_2(R)$; cf.\ (\ref{def:zentrum}). As usual, we write
$\PE_2(R):=\E_2(R)/H$ for the \emph{projective elementary group}.

\begin{lem}\label{lem:2fach}
The group $\E_2(R)$ acts $2$-$\dis$-\emph{transitively} on the
connected component $C\subseteq\PP(R)$, i.e.\ transitively on the set
of ordered pairs of distant points of $C$.
\end{lem}
\begin{proof}
Let $(p,q)$ be a pair of distant points in $C$. Since all points of
$C$ are in one orbit of $\E_2(R)$, there exists a matrix in $\E_2(R)$
sending $p$ to $R(1,0)$ and $q$ to $q_1\,\dis\, R(1,0)$. So
$q_1=R(t,1)$ for some $t\in R$. Now $E(0)E(-t)$ fixes $R(1,0)$ and
takes $q_1$ to $R(0,1)$.
\end{proof}

\para\label{para:unter}
Suppose that $R''$ is a subring of a ring $R'$. As $\GL_2(R'')$ is a
subgroup of $\GL_2(R')$, there is a mapping
\begin{equation*}
 \PP(R'')\to\PP(R'):R''(a'',b'')\mapsto R'(a'',b'')
\end{equation*}
which is easily seen to be injective. It will be used to identify
$\PP(R'')$ with a subset of $\PP(R')$. We have to distinguish between the
connected components $C''$ and $C'$ of $R'(1',0')$ in $\PP(R'')$ and
$\PP(R')$, respectively. Since $\E_2(R'')$ is a subgroup of
$\E_2(R')$, we have $C''\subseteq C'$.
\par
However, unless the centre $H''$ of $\E_2(R'')$ lies in the centre
$H'$ of $\E_2(R')$, the group $\PE_2(R'')$ cannot be considered as a
subgroup of $\PE_2(R')$: In general there are two matrices in $\E_2(R'')$
that induce the same transformation on $C''$, but distinct
transformations on $C'$. So we cannot always consider $\PE_2(R'')$ as
a group which acts on $C'$.
\par
The reader should keep these  remarks in mind with regard to the  following
theorem, which is our second main result:

\begin{thm}\label{thm:projektiv}
Let $\alpha:R\to R'$ be a Jordan homomorphism and denote by $R''$ the
subring of $R'$ generated by $R^\alpha$. Then the following
statements are true.
\begin{enumerate}
\item
The mapping
\begin{equation}\label{def:pe-gruppenabb}
  \alpha_\PE: \PE_2(R) \to \PE_2(R'') :
  H\cdot E(T)\mapsto H''\cdot E(T^\alpha),
\end{equation}
where $T\in\cS(R)$, is a well defined homomorphism of groups.
\item
Consider the connected component $C\subseteq\PP(R)$ and the connected
component $C''\subseteq\PP(R'')\subseteq\PP(R')$. Then the mapping
\begin{equation}\label{def:oa}
  \oa: C \to C'' :
  R(1,0)\cdot E(T)\mapsto R'(1',0')\cdot E(T^\alpha),
\end{equation}
where $T\in\cS(R)$, is well defined.
\item
The pair $(\alpha_\PE,\oa)$ is a homomorphism of transformation
groups.
\end{enumerate}
\end{thm}

\begin{proof}
(a) By Theorem \ref{thm:gruppen}~(b), $N_\alpha\subseteq H''$. So
there exists the canonical epimorphism
\begin{equation*}
 \eta:\E_2(R'')/N_\alpha
 \to
 \big(\E_2(R'')/N_\alpha\big)/(H''/N_\alpha)\cong \E_2(R'')/H''.
\end{equation*}
We identify its image with $\PE_2(R'')$. So, by
(\ref{def:gruppenabb}), the composition  $\alpha_\E\eta:
\E_2(R)\to\PE_2(R'')$ is a homomorphism. It follows from Theorem
\ref{thm:gruppen}~(a) that $H^{\alpha_\E} \subseteq H''/N_\alpha =
\ker\eta$. Hence $H\subseteq\ker\alpha_E\eta$ and $\alpha_\PE$ is a
well defined homomorphism of groups.
 \par
(b) We regard $\PE_2(R)$ and $\PE_2(R'')$ as transformation groups on
the connected components $C$ and $C''$, respectively.
 \par
Suppose that a matrix $E(S)$, $S\in\cS(R)$,  fixes $R(1,0)$.
Thus its first row has the form $(u,0)$ with $u\in R^*$. We infer
from Theorem \ref{thm:inv-0} that the first row of $E(S^\alpha)$
reads $(u',0')$ with a unit $u'\in R''$. So $E(S^\alpha)$ leaves
$R'(1',0')$ invariant. This means that under $\alpha_\PE$ the
stabilizer of $R(1,0)$ is mapped into the stabilizer of $R'(1',0')$.
\par
For each point $p\in C$ there is a sequence $T\in\cS(R)$ such that
$p=R(1,0)\cdot E(T)$. Also let $p=R(1,0)\cdot E(V)$ with
$V\in\cS(R)$. So the transformation $H\cdot E(T)E(V)^{-1}\in\PE_2(R)$
fixes $R(1,0)$, whence the transformation $H''\cdot
E(T^\alpha)E(V^\alpha)^{-1}\in\PE_2(R'')$ fixes $R'(1',0')$.
Therefore $\oa$ is well defined.
\par
(c) By (a) and (b), the diagram
%%neu-from
\begin{equation}\label{eq:diagramm}
\begin{diagram}
 C              & \rTo^{E(T)}        & C             \\
 \dTo^{\oa}     &                    & \dTo_{\oa}    \\
 C''            & \rTo^{E(T^\alpha)} & C''
\end{diagram}
\end{equation}
commutes for each $T\in\cS(R)$, whence the assertion follows.
\end{proof}
\para
For each point $p\in C$ there is a smallest integer $n\geq 0$ such
that $p=R(1,0)\cdot E(T)$ with $T\in R^n$. In fact, $n$ is just the
distance of $p$ and $R(1,0)$ in the distant graph \cite[formula
(10)]{blu+h-01a}. The supremum of all distances between points of $C$
is a non-negative integer or $\infty$. It is called the
\emph{diameter} of $C$. Furthermore, we have $(1,0)\cdot
E(0)^2=(-1,0)$ and $(1,0)\cdot E(t)=(1,0)\cdot E(1,t+1)$ for all
$t\in R$. So if the diameter of $C$ is finite, say $d$, then it is
enough to consider sequences $T\in R^m$ with fixed length
$m:=\max\{2,d\}$ in order to reach all points of $C$. By
%%neu-from
(\ref{eq:Ex}), in this case $\oa$ can be described by the single
formula
\begin{equation}\label{eq:explizit}
  R\big(\ele{1}{m}{T},\ele{1}{m-1}{T}\big)^\oa =
  R'\big(\ele{1}{m}{T^\alpha},\ele{1}{m-1}{T^\alpha}\big)
  \mbox { with } T\in R^m.
\end{equation}
This generalizes \cite[Theorem 2.4]{bart-89}, where $m=2$ and $R$ is
a ring of stable rank $2$. See also \cite[Remark 5.4]{blu+h-00y} for
the special case of an antiisomorphism of rings. We shall see in
Example \ref{exa:polynomring} that there are rings where one needs
sequences $T\in R^n$ for infinitely many $n\geq 0$ in order to
describe $\oa$.
\para
Let us state some immediate consequences of Theorem
\ref{thm:projektiv}. If $(a,b)$ and $(a',b')$ are the first rows of the
matrices $E(V)\in\E_2(R)$ and $E(V^\alpha)$, respectively, then
(\ref{eq:diagramm}) implies that
\begin{equation}\label{eq:oa}
  \big(R(a,b)\cdot E(T)\big)^{\oa}
  = R'(a',b')\cdot E(T^\alpha)
\end{equation}
for all $T\in\cS(R)$. Letting $E(V)=E(0)$ we get $(a,b)=(0,1)$ and
$(a',b')=(0',1')$. Therefore also the second rows of the matrices $E(T)$
and $E(T^\alpha)$, with $T\in\cS(R)$, represent points corresponding
under $\oa$. We deduce from (\ref{def:oa}), (\ref{eq:oa}), and
(\ref{eq:E-inv}) that
\begin{eqnarray}
  R(t,1)^\oa
  &=&
  \big(R(1,0)\cdot E(t)\big)^\oa
     = R'(1',0')\cdot E(t^\alpha)
     = R'(t^\alpha,1'),
  \label{eq:oa-affin_t1}
  \\
  R(1,t)^\oa
  &=&
  \big(R(0,1)\cdot E(0,-t,0)\big)^\oa
     = R'(0',1')\cdot E(0',-t^\alpha,0')
     = R'(1',t^\alpha),
  \label{eq:oa-affin_1t}
\end{eqnarray}
for all $t\in R$. In particular, $\oa$ is indeed an extension of the
mapping described in (\ref{eq:vorspann_affin}). Furthermore, $\oa$ is
a \emph{fundamental mapping}; this means that $R(1,0)^\oa=R'(1',0')$,
$R(0,1)^\oa=R'(0',1')$, and $R(1,1)^\oa=R'(1',1')$.
 \par
If $R^\alpha=R''$ then each point of $C''$ can be written as
$R'(1',0')\cdot E(T^\alpha)$, whence $C^\oa=C''$. Similarly, if
$R^\alpha=R'$ then $C^\oa=C'=C''$. If $\oa$ injective then
(\ref{eq:oa-affin_t1}) implies the injectivity of $\alpha$. Also, if
$\alpha$ is bijective then $\alpha^{-1}:R'\to R$ is a Jordan
isomorphism and $\overline{\alpha^{-1}}$ is easily seen to be the
inverse of $\oa$. What remains open here is whether or not
$C^\oa=C''$ implies that $R^\alpha=R''$ and whether or not $\oa$ is
injective if $\alpha$ is injective.
 \para\label{para:harmonic}%%NEU!!
A mapping $\PP(R)\to\PP(R')$ is said to be \emph{harmonic} if it
preserves harmonic quadruples or, in other words, if it preserves
cross ratio $-1$ \cite[1.3.5]{herz-95}. If $(p_0,p_1,p_2,p_3)$ is a
harmonic quadruple in $\PP(R)$ then $p_0$, $p_1$, and $p_i$ are
mutually distant for $i\in\{2,3\}$. Furthermore, all four points are
mutually distant if and only if $-1\neq 1$ in $R$. Otherwise $p_2=p_3$,
whence in this case harmonic mappings do not deserve our interest.
 \par
In order to state the next result we have to allow the domain and the
codomain of a harmonic mapping to be a subset of a projective line.

\begin{prop}\label{prop:dist+harm}
Let $\alpha:R\to R'$ and $\oa:C\to C''$ be given as in Theorem
\ref{thm:projektiv}. Then the following statements are true:
\begin{enumerate}
 \item
$\oa$ takes pairs of distant points to pairs of distant points.
 \item
$\oa$ is a harmonic mapping.
\end{enumerate}
\end{prop}
\begin{proof}
(a) Let $(p,q)$ be a pair of distant points in $C$. By Lemma
\ref{lem:2fach}, there exists a matrix in $\E_2(R)$ sending $p$ to
$R(1,0)$ and $q$ to $R(0,1)$. Clearly, the $\oa$-images of $R(1,0)$
and $R(0,1)$ are distant and, by (\ref{eq:diagramm}), the points
$p^\oa$ and $q^\oa$ are distant, too.
\par
(b) Suppose that $p_1,p_2$ are points of $\PP(R)$. Then
$\big(R(1,0),R(0,1),p_1,p_2\big)$ is a harmonic quadruple if and only if
there is a $u\in R^*$ with $p_1=R(u,1)$ and $p_2=R(-u,1)$.
\par
If we are given a harmonic quadruple in $C$ then, by Lemma
\ref{lem:2fach} and (\ref{eq:diagramm}), we may assume without loss
of generality that the first two points are $R(1,0)$ and $R(0,1)$. So
the remaining two points can be described as above. We deduce from
(\ref{eq:oa-affin_t1}) and $u^\alpha\in{R'}^*$ that under $\oa$ a
harmonic quadruple is obtained.
\end{proof}

The first part of Proposition \ref{prop:dist+harm} implies that
\begin{equation}\label{eq:kontrahierend}
  \dist(p^\oa,q^\oa)\leq \dist(p,q) \mbox{ for all } p,q\in C.
\end{equation}
Here $\dist(p,q)$ denotes the distance of two points in the distant
graph.
 \para
We turn to following question: If $C\neq\PP(R)$, how should one extend
$\oa$ to a mapping $\gamma:\PP(R)\to\PP(R')$? In view of Proposition
\ref{prop:dist+harm} such an extension should at least preserve
distant pairs and harmonic quadruples. Here are solutions to this
problem.
\begin{exas}\label{exa:fortsetzung}
\begin{enumerate}
\item
For each connected component $C_\mu$ of $\PP(R)$ other than $C$
choose a matrix $A_\mu\in\GL_2(R)$ such that the first row of $A_\mu$
represents a point of $C_\mu$. Also select a matrix
$A'_\mu\in\GL_2(R')$. Then $A_\mu^{-1}$ maps $C_\mu$ onto $C$, $\oa$
takes $C$ into $C'$, and $A'_\mu$ maps $C'$ into some connected
component of $\PP(R')$. In this way we obtain a solution $\gamma$ by
pasting together all these mappings.
\item
In Example (a) the matrices $A'_\mu$ can be chosen at random. Suppose
now that there is a homomorphism $\sigma:\GL_2(R)\to\GL_2(R')$ such
that $\sigma$ restricts to $\alpha_\E$ and such that $\sigma$ takes
the stabilizer of $R(1,0)$ into the stabilizer of $R'(1',0')$. Then
\begin{equation}\label{eq:sigma}
  \overline\sigma : \PP(R)\to\PP(R') :
  R(1,0)\cdot M\mapsto R'(1',0')\cdot M^\sigma\;\;
  \mbox{with } M\in\GL_2(R)
\end{equation}
is a well defined extension of $\oa$ and $(\sigma,\overline\sigma)$
is a homomorphism of group actions. The mapping $\overline\sigma$
fits into Example (a) by choosing $A'_\mu=A_\mu^\sigma$.
\par
The homomorphisms $\alpha_*$, $\alpha_{**}$, and $\beta$ which have
been introduced in \ref{exa:jordan} (a), (b), and (c), respectively,
satisfy the conditions  above: A matrix $M\in\GL_2(R)$ stabilizes
$R(1,0)$ if and only if there are elements $a,d\in R^*$ and $c\in R$ with
$M=\SMat2{a & 0\\ c & d}$. If $\alpha$ is a homomorphism of rings
then the assertion is immediate. Furthermore, in this case we get the
well known mapping
\begin{equation}\label{eq:homo}
  \overline{\alpha_*}:\PP(R)\to\PP(R') : R(a,b)\mapsto R'(a^\alpha,b^\alpha).
\end{equation}
If $\alpha$ is an antihomomorphism then a straightforward calculation
shows
\begin{equation*}
 M^{\alpha_{**}}
 =
 \Mat2{(d^\alpha)^{-1}                        & 0'   \\
       (a^\alpha)^{-1}c^\alpha(d^\alpha)^{-1} & (a^\alpha)^{-1}},
\end{equation*}
whence the assertion follows also in the remaining cases, cf.\ also
\cite[Remark 5.4]{blu+h-00y}.
\end{enumerate}
\end{exas}
We end this section with an example where $\PP(R)$ has more than one
connected component, the connected components of $\PP(R)$ have
infinite diameter, and $\alpha$ is a proper Jordan endomorphism.

\begin{exa}\label{exa:polynomring}
We specify and slightly modify the data of Example
\ref{exa:jordan}~(e) as follows: Let $D=K$ be a commutative field,
let $B=K[x,y]$ be the algebra of polynomials in commuting
indeterminates $x,y$ over $K$, and let $\chi:K[x,y]\to K: f\mapsto
f(0,0)$. The module $M=K^3$ and its multiplication remain unchanged.
Hence $R=K[x,y]\oplus K^3$. Again $\alpha_1=\id_{K[x,y]}$, but now
$\alpha_2$ is chosen to be any $K$-linear mapping with $\eps_3$ not being
an eigenvector. Then $\alpha$ is a proper $K$-linear Jordan
endomorphism, since $(\eps_1\eps_2)^\alpha=\eps_3^\alpha\notin
K\eps_3$, whereas $\eps_1^\alpha\eps_2^\alpha\in K\eps_3$ and
$\eps_2^\alpha\eps_1^\alpha\in K\eps_3$.
\par
The projection $\pi:R\to K[x,y]$ is an epimorphism of $K$-algebras
with kernel $\{0\}\oplus K^3$. By (\ref{eq:homo}), it gives rise to a
mapping $\overline{\pi_*}:\PP(R)\to \PP(K[x,y])$ which is surjective
\cite[Proposition 3.5~(3)]{blu+h-00b}. Under $\overline{\pi_*}$ the
connected component of $R(1,0)$ is mapped onto the connected
component of $K[x,y](1,0)$, since $\pi_\PE$ is a surjection of
$\PE_2(R)$ onto $\PE_2(K[x,y])$. The projective line $\PP(K[x,y])$
has more than one connected component and all its connected
components have infinite diameter \cite[Example 5.7~(c)]{blu+h-01a}.
So (\ref{eq:kontrahierend}) implies that also $\PP(R)$ has more than
one connected component and that $C\subset\PP(R)$ has infinite
diameter. By \cite[Theorem 3.2~(a)]{blu+h-01a}, then all connected
components of $\PP(R)$ have infinite diameter.
\end{exa}

\section{On homomorphisms of chain geometries}\label{se:homomorph}
 \para\label{para:kettengeom}%%NEU!!
If $K\subseteq R$ is a (not necessarily commutative) subfield, then
the projective line $\PP(K)$ can be identified with a subset of
$\PP(R)$; cf.\ \ref{para:unter}. The orbit of $\PP(K)$ under the
group $\GL_2(R)$ is the set of \emph{$K$-chains}. It turns $\PP(R)$
into a {\em chain geometry\/} $\Sigma(K,R)$. The following basic
properties of chain geometries can be found in \cite{blu+h-00a}: Any
three mutually distant points are on at least one $K$-chain. Two
distinct points are distant if and only if they are on a common
$K$-chain. Therefore each $K$-chain is contained in a unique
connected component. In contrast to \cite{herz-95} it is not assumed
that $K$ is in the centre of $R$, whence in \cite{blu+h-00a} we used
the term \emph{generalized chain geometry} for what is here called a
chain geometry.
\par
We now consider two chain geometries $\Sigma(K,R)$, $\Sigma(K',R')$
and the mapping (\ref{def:oa}). The following result is a
generalization of \cite[Theorem 2.4]{bart-89} and
\cite[9.1]{herz-95}:

\begin{thm}\label{thm:ketten}
Let $\alpha:R\to R'$ be a Jordan homomorphism. The mapping $\oa:C\to
C''$ maps $K$-chains into $K'$-chains if and only if for each $c\in
R^*$ there is a $u_c'\in {R'}^*$ such that
\begin{equation}\label{eq:ketten}
  (Kc)^\alpha \subseteq ({u_c'}^{-1}K'u_c') c^\alpha.
\end{equation}
\end{thm}
\begin{proof}
For each $c\in R^*$ the point set
\begin{equation}\label{def:D_c}
  \cD_c:=\{R(k c,1)\mid k\in K\}\cup\{R(1,0)\}
\end{equation}
is a $K$-chain through $R(1,0)$, $R(0,1)$, and $R(c,1)$. The
$K'$-chains passing through $R'(1',0')$, $R'(0',1')$, and
$R'(c^\alpha,1')$ are exactly the sets
\begin{equation}\label{eq:alle_ketten}
    \{R\big(({u'}^{-1}k' u') c^\alpha ,1'\big)
 \mid k'\in K'\}\cup\{R'(1',0')\}
\end{equation}
where $u'$ ranges in ${R'}^*$.
\par
Suppose that $\oa$ maps $K$-chains into $K'$-chains. So for each
$c\in R^*$ the point set $\cD_c^\oa$ is a subset of a $K'$-chain. Now
(\ref{eq:oa-affin_t1}) implies that this chain contains the points
$R'(1',0')$, $R'(0',1')$, and $R'(c^\alpha,1')$, whence it can be
written in the form (\ref{eq:alle_ketten}) for some $u'_c\in {R'}^*$
depending on $c$. Applying (\ref{eq:oa-affin_t1}) to each point of
(\ref{def:D_c}) shows that condition (\ref{eq:ketten}) is satisfied.
\par
Conversely, (\ref{eq:ketten}) forces that each $K$-chain $\cD_c$
given by (\ref{def:D_c}) is mapped into one of the $K'$-chains given
by (\ref{eq:alle_ketten}). By Lemma \ref{lem:2fach}, every $K$-chain
%%neu-from
$\cD\subseteq C$ is $\E_2(R)$-equivalent to some $K$-chain through
$R(1,0)$ and $R(0,1)$. Such a chain has the form
\begin{equation*}
  \{R(ka,b)\mid k\in K\}\cup\{R(1,0)\} \mbox { with } a,b\in R^*.
\end{equation*}
Since $\diag(b,b^{-1}) = E(-b)E(-b^{-1})E(-b) \in \E_2(R)$, the
chains $\cD$ and $\cD_c$, where $c:=ab$, are in one orbit of
$\E_2(R)$. Now (\ref{eq:diagramm}) shows that also $\cD^\oa$ is a
subset of a $K'$-chain. %%neu a subset
\end{proof}

Condition (\ref{eq:ketten}) reduces to $(Kc)^\alpha\subseteq
K'c^\alpha$ provided that $K'$ is invariant under all inner
automorphisms of $R'$. This is the case whenever $K'$ is in the
centre of $R'$, but there are also other possibilities \cite[Examples
2.5]{blu+h-00a}.
 \para\label{para:homo}%%NEU!!
We close with some remarks on a mapping $\oa$ where $\alpha$
satisfies the conditions of Theorem \ref{thm:ketten}. If $\PP(R)=C$
then $\oa$ is a \emph{homomorphism of chain geometries}, i.e.,
$K$-chains are mapped into $K'$-chains. If $\PP(R)\neq C$ then $\oa$
can be extended to a mapping $\gamma:\PP(R)\to\PP(R')$ according to
Example \ref{exa:fortsetzung}~(a). As the general linear group
preserves chains, any such $\gamma$ is a homomorphism of chain
geometries. Explicit examples for this latter case arise from Example
\ref{exa:polynomring}, because all Jordan endomorphisms described
there are $K$-linear and thus fulfil condition (\ref{eq:ketten}).
%%neuer Satz

%%\bibliographystyle{plain}
%%\bibliography{ketten}

\begin{thebibliography}{10}
\bibitem{bart-89}
Bartolone C (1989)
\newblock Jordan homomorphisms, chain geometries and the fundamental theorem.
\newblock {Abh Math Sem Univ Hamburg} \textbf{59}: 93--99

\bibitem{bart+b-85}
Bartolone C, Bartolozzi F (1985)
\newblock Topics in geometric algebra over rings.
\newblock In: Kaya R, Plaumann P, Strambach K (eds) {Rings and
  Geometry}, pp 353--389. Dordrecht: Reidel

\bibitem{bart+f-79}
Bartolone C, Di Franco F (1979)
\newblock A remark on the projectivities of the projective line over a
  commutative ring.
\newblock {Math Z} \textbf{169}: 23--29

\bibitem{benz+s+s-81}
Benz W, Samaga H-J,  Schaeffer H (1981)
\newblock Cross ratios and a unifying treatment of von {S}taudt's notion of
  reeller {Z}ug.
\newblock In: Plaumann P, Strambach K (eds) {Geometry -- von
  Staudt's Point of View}, pp 127--150. Dordrecht: Reidel

\bibitem{blunck-92b}
Blunck A (1992)
\newblock Chain geometries over local alternative algebras.
\newblock {J Geom} \textbf{44}: 33--44

\bibitem{blunck-94}
Blunck A (1994)
\newblock Chain spaces over {J}ordan systems.
\newblock {Abh Math Sem Univ Hamburg} \textbf{64}: 33--49

\bibitem{blu+h-00a}
Blunck A,  Havlicek H (2000)
\newblock Extending the concept of chain geometry.
\newblock {Geom Dedicata} \textbf{83}: 119--130

\bibitem{blu+h-00b}
Blunck A,  Havlicek H (2000)
\newblock Projective representations {I.} {P}rojective lines over
rings.
\newblock {Abh Math Sem Univ Hamburg} \textbf{70}: 287--299

\bibitem{blu+h-01a}
Blunck A,  Havlicek H (2001)
\newblock The connected components of the projective line over a ring.
\newblock {Adv Geom} \textbf{1}: 107--117

\bibitem{blu+h-00y}%%NEU!!
Blunck A,  Havlicek H (2001)
\newblock The dual of a chain geometry.
\newblock {J Geom} \textbf{72}: 27--36

\bibitem{cirl+e-90}
Cirlincione L, Enea MR (1990)
\newblock Una generalizzazione del birapporto sopra un anello.
\newblock {Rend Circ Mat Palermo (II)} \textbf{34}: 271--280

\bibitem{cohn-66}
Cohn PM (1966)
\newblock On the structure of the {${\rm GL}_2$} of a ring.
\newblock {Inst Hautes Etudes Sci Publ Math} \textbf{30}: 365--413

\bibitem{ferr-81}
Ferrar JC (1981)
\newblock Cross-ratios in projective and affine planes.
\newblock In: Plaumann P, Strambach K (eds) {Geometry -- von
  Staudt's Point of View}, pp 101--125. Dordrecht: Reidel

\bibitem{herz-87a}
Herzer A (1987)
\newblock On isomorphisms of chain geometries.
\newblock {Note Mat} \textbf{8}: 251--270

\bibitem{herz-95}
Herzer A (1995)
\newblock Chain geometries.
\newblock In: Buekenhout F (ed) {Handbook of Incidence Geometry},
  pp 781--842.  Amsterdam: Elsevier

\bibitem{hua-53}
Hua LK (1953)
\newblock On semi-homomorphisms of rings and their application in projective
  geometry ({R}ussian).
\newblock {Uspehi Matem Nauk (NS)} \textbf{8}: 143--148

\bibitem{jac-68}
Jacobson N (1968)
\newblock {Structure and Representation of Jordan Algebras}.
\newblock Providence: Amer Math Soc

\bibitem{kar+k-88}
Karzel H, Kroll H-J (1988)
\newblock {Geschichte der Geometrie seit Hilbert}.
\newblock Darmstadt: Wiss Buchges

\bibitem{lam-91}
Lam TY (1991)
\newblock {A First Course in Noncommutative Rings}.
\newblock New York: Springer

\bibitem{lima+l-77b}
Limaye BV, Limaye NB (1977)
\newblock The fundamental theorem for the projective line over commutative
  rings.
\newblock {Aequationes Math} \textbf{16}: 275--281

\bibitem{lima+l-77a}
Limaye BV, Limaye NB (1977)
\newblock Fundamental theorem for the projective line over non-commutative
  local rings.
\newblock {Arch Math (Basel)} \textbf{28}: 102--109

\end{thebibliography}

Authors' addresses:
 \\
Andrea Blunck, Fachbereich Mathematik, Universit\"at Hamburg,
Bundesstra{\ss}e 55, D--20146 Hamburg, Germany
 \\
email: andrea.blunck@math.uni-hamburg.de
 \\
Hans Havlicek, Institut f\"ur Geometrie, Technische Universit\"at,
Wiedner Hauptstra{\ss}e 8--10, A--1040 Wien, Austria
 \\
email: havlicek@geometrie.tuwien.ac.at
\end{document}